\documentclass[11pt]{article}
\font\smallit=cmti10

\usepackage{amsmath,amssymb,amsbsy,amsfonts,amsthm,latexsym}
\usepackage{amsopn,amstext,amsxtra,euscript,amscd}
\usepackage[all]{xy}
\usepackage{cite}
 \makeatletter

\newtheorem*{lemma*}{Lemma}
\newtheorem*{proposition*}{Proposition}
\newtheorem*{corollary*}{Corollary}
\newtheorem*{definition*}{Definition}
\newtheorem{theorem}{Theorem}
\newtheorem{lemma}[theorem]{Lemma}
\newtheorem{claim}[theorem]{Claim}
\newtheorem{corollary}[theorem]{Corollary}
\newtheorem{proposition}[theorem]{Proposition}
\newtheorem{definition}{Definition}
\newtheorem{question}[theorem]{Open Question}
\newtheoremstyle{dotless}{}{}{\itshape}{}{\bfseries}{}{ }{}
\newtheorem*{nothing}{}
\def\cH{{\mathcal H}}
\def\cA{{\mathcal A}}
\def\cS{{\mathcal S}}
\def\cK{{\mathcal K}}
\def\cF{{\mathcal F}}

\def\cL{{\mathcal L}}
\def\cD{{\mathcal D}}
\def\R{{\mathbb R}}
\newcommand{\C}{\mathbb C}
\newcommand{\N}{\mathbb N}
\newcommand{\Z}{\mathbb Z}

\def\gcm{\mathop{gcm}}
\def\ln#1{\langle\langle #1\rangle\rangle}
\def\ls#1{\langle #1 \rangle}
\def\FP{\mbox{\cal FP}}
\def\ignor#1{}
\begin{document}

\begin{center}
\uppercase{\bf Approximations of groups, characterizations of sofic groups, 
and equations over groups.}
\vskip 20pt
{\bf Lev Glebsky}\\
{\smallit Instituto de Investigaci{\'o}n en Comunicaci\'on {\'O}ptica    
          Universidad Aut{\'o}noma de San Luis Potos{\'i} Av. Karakorum 1470, Lomas 4a 78210
          San Luis Potosi, Mexico}\\
{\tt glebsky@cactus.iico.uaslp.mx}\\ 
\end{center}

\begin{abstract}
We give new characterizations of sofic groups:

-- A group $G$ is sofic if and only if  it is a subgroup of  a quotient of a direct product of alternating or symmetric groups.

-- A group $G$ is sofic if and only if any system of equations solvable in all alternating groups is solvable over $G$.

The last characterization allows to express soficity of an existentially closed group by $\forall\exists$-sentences. 

\end{abstract}

\section{Introduction} 

Sofic groups have been defined\footnote{Definition 3.5 of \cite{AGG} is the definition of sofic 
groups where the authors missed the separability conditions. All groups trivially satisfy 
this definition.} in \cite{Gr99, W00} 
in relation with the Gottschalk surjunctivity conjecture.
Hyperlinear groups have been introduced in \cite{Ra08} in relation with 
Connes' embedding conjecture.  
It is known that sofic groups are hyperlinear,
the reverse inclusion is an open question. 

Some famous group theory conjectures   
(Kervaire-Laudenbach, Gottschalk,
Connes' embedding conjectures) are established for sofic groups,
see \cite{Lupini1, Pestov1} and the references therein. It is also known that some important classes of groups are sofic,
for example, amenable, residually amenable, extensions of sofic groups by amenable groups, etc.,
\cite{EL, Lupini1, Pestov1}. 
An open question is whether or not all groups are sofic (respectively, hyperlinear).

Classically, sofic (resp. hyperlinear) groups are defined as being  metric approximable  by
symmetric groups (resp. finite-dimensional unitary groups), \cite{Thom1}. 
It is possible to define metric approximations by
different  classes of groups, see \cite{Arzhan2} and Definition~\ref{def_m-approx} 
of the present paper. 
We call them $(\cK,\cL)$-approximations, where $\cK$ is a class of groups and 
$\cL$ is a class of invariant length functions on these groups. 
Choosing the $\cK$ and $\cL$
one defines different classes of groups as $(\cK,\cL)$-approximable groups. 
For example, weakly sofic groups, \cite{GR}, 
and linear sofic groups, \cite{Arzhan1}, for $(\cK,\cL)=(\{\mbox{all finite groups}\},\{\mbox{all length functions}\})$ and
$(\cK, \cL)=(\{GL_n(\C)\},\{\mbox{normalized rank}\})$, respectively.
The following result  is classical here, see Proposition~\ref{prop_ultr} for details.
\begin{nothing}
 A group $G$ is 
$(\cK,\cL)$-approximable if and only if $G$ is isomorphic to a subgroup
of a metric ultraproduct with respect to $\cL$ of groups from $\cK$. 
\end{nothing}
By definition, the metric approximation
depends on invariant length functions and a class of groups. 
The structure of the set of invariant length functions on a group depends 
on the algebra of 
the conjugacy classes of this group. Indeed, an invariant length function $\|\cdot\|$, see Definition~\ref{def_len},
is constant along conjugacy classes and satisfies inequality $\|g_3\|\leq \|g_1\|+\|g_2\|$ for 
$g_3\in C(g_1)C(g_2)$, where $C(g)$ is the conjugacy class of $g$.

In the present article we 
investigate the notion of approximation based on products of conjugacy classes 
without direct use of length functions, see Definition~\ref{def_w-approx}. 
Such approximations will be called $\cK$-approximations\footnote{Some other types 
of approximations are defined in \cite{AGG,Arzhan2}.}. A group possessing $\cK$-approximation will
be called $\cK$-approximable group.
Let $Sym$, $Alt$, $Nil$, $Sol$, $Fin$ be the classes of finite symmetric,
finite alternating, finite nilpotent, finite solvable and all finite groups, 
respectively. We show that the classes of $Alt$-approximable groups, 
$Sym$-approximable groups, and sofic groups coincide. 
$Fin$-approximable groups are called weakly sofic \cite{GR}. 

Proposition~\ref{prop_ultr} with Lemma~\ref{lm_w_and_m} implies that a
$\cK$-approximable group is isomorphic to a subgroup of a quotient of an
(unrestricted) direct product of groups from $\cK$. 
Is it true that any quotient of
any direct product of groups from $\cK$ is $\cK$-approximable?
We don't know the answer to the question but we answer it  affirmatively  if 
$\cK$ is a class of compact groups satisfying
$\cK=\prod(\cK)$, where $\prod(\cK)$ 
is a class of all finite 
direct products of groups from $\cK$. The question has an affirmative answer for $\cK=Sym$ and $Alt$ also. 
We start by proving that for a class of compact
groups $\cK$ a quotient of a  direct product $G=\prod\limits_{i\in\N} H_i$, $H_i\in \cK$,
is $\prod(\cK)$-approximable, see Proposition~\ref{prop_direct}. 
Then we show that $\cK.approx=\prod(\cK).approx$ for
$\cK=Alt$ and $Sym$, Proposition~\ref{sofic_without_metric}. 
The property $\cK.approx=\prod(\cK).approx$ allow us to give the
following characterization of $\cK$-approximable groups.
\begin{nothing}
Let $F=\ls{a_1,\dots,a_r}$ be a finitely generated free group and $N\triangleleft F$. 
Then $F/N$ is $Nil$, $Sol$, $Fin$-approximable if and only if
$N=\langle\langle N\rangle\rangle_{\hat F}\cap F$. 
Here $F\hookrightarrow \hat F$ and
$\hat F$ is the pro-nilpotent, pro-solvable, pro-finite completion of $F$, 
respectively.
$\langle\langle N\rangle \rangle_{\hat F}$ denotes the minimal normal subgroup
of the corresponding $\hat F$ containing $N$. (Notice, that $\langle\langle N\rangle \rangle_{\hat F}$ need 
not to be topologically closed.)
\end{nothing}

The similar characterization for sofic groups has a peculiarity due to 
the fact that
$Alt$ and $Sym$ are not closed under taking subgroups. Still, one can define 
$F\hookrightarrow \cA$ (resp. $F\hookrightarrow \cS$) in such a way that $F/N$ is
sofic if and only if $N=\ln{N}_\cA\cap F$ (resp. $N=\ln{N}_\cS\cap F$). The group 
$\cA$, for example, is a direct product of alternating groups,
every alternating group $A_m$ appears finitely many times, one for each 
homomorphism $\phi:F\to A_m$, 
with the inclusion $F\hookrightarrow\cA$ naturally
defined by these $\phi$'s.  Notice, that in contrast to the above statement, the closure
of $F$ (in the product Tychonoff topology) is not the whole $\cA$. In fact, this closure is
isomorphic to the pro-finite completion of $F$. The construction of $\cS$ is similar the only difference that instead of 
alternating groups $A_m$ one should take symmetric groups $S_m$. 

We may reformulate the above characterization using equations over groups.
Let $\bar a=(a_1,\dots,a_r)$ and $\bar x=(x_1,\dots,x_k)$ be the symbols for constants 
and variables, respectively. 
$|\bar a|=r$ and $|\bar x|=k$. Let $\ls{\bar a,\bar x}$ be a free group freely  generated by
$\bar a$ and $\bar x$. Let
$w_i\in\ls{\bar a,\bar x}$ and 
$\bar w=(w_1,\dots, w_n)$. 
Consider the system of equations $\bar w=1$ ($w_1=1,\dots, w_n=1$). 
\begin{definition}\label{def_sys}
We say that $\bar w $ is solvable in a group $G$ 
if the sentences 
$$
\forall \bar a \exists \bar x \bigwedge_{i=1}^nw_i(\bar a,\bar x)=1 
$$  
is valid in $G$.  
We say that a system $\bar w$ is solvable over group
$G$ if for some $H>G$ 
$$
\forall \bar a\in G^r\; \exists \bar x\in H^k\; \bigwedge_{i=1}^nw_i(\bar a,\bar x)=1 
$$
\end{definition} 
Denote by $Sys(G)$ the set of all finite systems of equations solvable in $G$.  For a class $\cK$ let 
$Sys(\cK)=\bigcap\limits_{G\in\cK} Sys(G)$.
In Section~\ref{sec_equation} the following proposition is proved.
\begin{proposition*}
Let $\cK=Nil, Sol, Fin, Alt$ or $Sym$. Then a group $G$
is $\cK$-approximable if and only if all systems $\bar w \in Sys(\cK)$ 
are solvable over $G$.
\end{proposition*}
{\bf Remark.} 
\begin{itemize}
\item It is not clear if the sets $Sys(\cK)$ are recursive.     
\item Let $F$ be a noncommutative  free group. By the solution of Tarski problem 
$Sys(F)$ is independent of $F$ and recursive, \cite{Khar_Myas, Sela}. It easily follows that
$Sys(F)\subset Sys(G)$ for any group $G$. So, for our purpose $Sys(F)$ are ``trivial equations'' and it suffices to consider 
$Sys'(\cK)=Sys(\cK)\setminus Sys(F)$.   
\item $Sys(Fin)\neq\emptyset$, see \cite{Coulbois}.
\end{itemize}
Although the characterization looks impractical, we can derive some consequences.

 Paul Schupp suggests 
that existentially closed groups are non sofic \cite{Gordon}. The existentially closed groups are analogues of algebraically closed 
fields, see \cite{Higman}. 
The following proposition
supports this hypothesis.
\begin{proposition*}
Let $\cK=Nil, Sol, Alt, Sym$, or  $Fin$. An existentially closed group 
is $\cK$-approximable if
and only if all groups are $\cK$-approximable.
\end{proposition*}
\begin{proof}
An existentially closed group $G$ is $\cK$ -approximable if and only if the sentences
$$
\forall \bar a \exists \bar x \bigwedge_{w\in\bar w}w(\bar a, \bar x)
$$
are valid in $G$ for any $\bar w\in Sys(\cK)$. A $\forall\exists$-sentence 
has the same truth
value in every existentially closed group (Corollary 9.6, page 121, \cite{Higman}).
So, all existentially closed groups simultaneously either $\cK$-approximable or not. 
If there
were a non $\cK$-approximable group, then  
a finitely generated non $\cK$-approximable group  would exist. ($\cK$-approximability 
is a local property.) 
One may embed this group into an existentially closed group $E$ (\cite{Higman}, page 6 and proof of Theorem~3.10). 
As $\cK$-approximability is closed under taking subgroups, $E$ is not $\cK$-approximable.
So, all existentially 
closed groups are non $\cK$-approximable.
\end{proof} 
Another application is related to
algebraic groups over algebraically closed fields. The following statement is Proposition~\ref{prop_Fin->AGCF} of
Section~\ref{sec_algebraic_groups}.
\begin{nothing}
Any  $\bar w\in Sys(Fin)$ is solvable in any algebraic group over an algebraically 
closed field.
\end{nothing}
It implies 
\begin{nothing}
Let $\cK$ be a class of algebraic groups over algebraically closed fields (fields may be different for different groups from $\cK$). 
Then 
$\cK$-approximable groups are $Fin$-approximable, that is, weakly sofic. 
\end{nothing}
This corollary is a generalization of a result of 
\cite{Arzhan1}: ``all linear sofic groups are weakly sofic''.

A subgroup $H<G$ has the congruence extension property (CEP) 
if any $N\triangleleft H$ satisfies
$N=H\cap \ln{N}_G$. A subgroup $H<G$ almost has CEP if there exists a finite set 
$\cF\not\ni 1$ such that any $N\triangleleft H$, $N\cap \cF=\emptyset$, satisfies
$N=H\cap \ln{N}_G$. Some authors say that $H$ is a normal convex subgroup of $G$ if
$H<G$ has the CEP, \cite{Stal, Brick1, Brick2}. A group $G$ is called SQ-universal if
any countable group injects into a quotien of $G$. 

\begin{proposition*}[{\bf Some equivalence}]
Let $F=\ls{x,y}$ be a free group of rank 2. Let $\cK=Nil, Sol, Fin, Alt, Sym$ and
$F\hookrightarrow\hat F_\cK$ be pro-nilpotent, pro-solvable, pro-finite completion of $F$, 
$\cA$ or $\cS$, 
respectively.
Then the following are equivalent.
\begin{enumerate}
\item All groups are $\cK$-approximable.
\item $F$ satisfies CEP in $\hat F_\cK$.
\item $F$ almost satisfy CEP in $\hat F_\cK$.
\item $\hat F_\cK$ is SQ-universal.
\end{enumerate} 
\end{proposition*}
Similar proposition was proven by Goulnara Arzhantseva, Jakub Gismatullin and others 
\cite{Gis}. 
 It follows from construction of \cite{Howie}  that $F$ does not possess CEP in it's pro-nilpotent completion. 
In fact, results of \cite{Howie}  
imply that any finitely generated perfect group is not $Nil$-approximable. 
The question ``if all groups are $Sol$-approximable'' seems to be open. 

We use the following notations:
$F<G$ denotes ``$F$ is a subgroup of $G$''. $F\triangleleft G$ denotes ``$F$ 
is a normal subgroup of $G$''.
Let $X\subseteq F$ and $F<G$. Then $\ls{X}$ denotes ``the subgroup generated by $X$''; 
$\ln{X}_F$ denotes ``the normal subgroup of $F$, generated by $X$''. 
If a group containing $\bar a=(a_1,\dots,a_k)$ is not assumed then
$\ls{\bar a}$ denotes the free group freely generated by $\bar a$.
As usual, $X\subset_{fin}Y$ means ``$X$ is a finite subset of $Y$''. 
For two sets $A$ and $B$ let $A\setminus B=\{a\in A\;|\;a\not\in B\}$.
Given $N\triangleleft G$, as usual,  $G/N$ denotes the quotient of $G$ by $N$.
\section{Approximation of groups}
\subsection{Metric approximation}
In this subsection we define metric approximation, the $(\cK,\cL)$-approximation, of groups.
We follow the lines of \cite{Arzhan2, Lupini1, Pestov1, Thom1}. 
\begin{definition}\label{def_len}
Let $G$ be a group. An invariant (pseudo) length function is a map 
$\|\cdot\|:G\to [0,\infty [$ such that $\forall g,h\in G$
\begin{itemize}
\item $\|1\|=0$
\item $\|gh\|\leq \|g\|+\|h\|$
\item $\|h^{-1}gh\|=\|g\|$ ($\|\cdot\|$ is invariant within a conjugacy class)
\end{itemize}
\end{definition}
(This is the definition of pseudo length function of \cite{Thom1}. The difference  
with the standard definition of the length function is not essential here, 
so we use just 
``length function''.)

\begin{definition}\label{def_m-approx}
Let $\cK$ be a class of groups and $\cL$ be a class of invariant length functions 
on groups from $\cK$. A group $G\in\cK$ may have several length functions in $\cL$.
Let $\cL_G$ be the set of length functions on $G$ in $\cL$.
All length functions are denoted by $\|\cdot\|$. 
We say that a group $G$ is $(\cK,\cL)$-approximable if 
\begin{itemize}
\item there exists $\alpha:G\to\R$, $\alpha_1=0$ and $\alpha_g>0$ for $g\neq 1$
\item for any $\Phi\subset_{fin} G$, for any $\epsilon>0$ there exist a function 
$\phi:\Phi\to H\in\cK$ and $\|\cdot\|\in\cL_H$  such that 
\begin{itemize}
\item $\phi(1)=1$
\item $\|\phi(g)\|\geq \alpha_g$ for any $g\in\Phi$
\item $\|\phi(gh)(\phi(g)\phi(h))^{-1}\|<\epsilon$ for any $g,h,gh\in\Phi$.
\end{itemize}
Let $(\cK,\cL).approx$ denote the class of $(\cK,\cL)$-approximable groups.
\end{itemize}
\end{definition}
Let $\omega$ be a non-principal ultrafilter  over $\N$, $H_i\in\cK$, 
and $\|\cdot\|_i\in\cL_{H_i}\subset\cL$. Let
$N\triangleleft\prod\limits_{i=1}^\infty H_i$ be defined as
$$
(h_1,h_2,\dots)\in N\;\;\Longleftrightarrow\;\;\lim_{\omega} \|h_i\|_i=0.
$$ 
Denote $\prod_\omega H_i=\prod_i H_i/N$, the metric ultraproduct of $H_i$ 
(with respect to $\|\cdot\|_i$ and $\omega$).
The following characterization of $(\cK,\cL).approx$ is standard \cite{Thom1, Lupini1, Pestov1}.
\begin{proposition}\label{prop_ultr}
$G\in (\cK,\cL).approx$ if and only if there exists an injection 
$G\hookrightarrow \prod_\omega H_n$ for some non principal ultrafilter $\omega$,
a sequence $H_1,H_2,\dots \subset \cK$, and a sequence of length functions
$\|\cdot\|_i\in\cL_{H_i}$.    
\end{proposition}
\begin{definition}\label{def_metric_separ}
Let $N\triangleleft G$. We say that $N$ is $(\cK,\cL)$-separated normal subgroup of $G$
if for some $\alpha:G\setminus N\to \R^+=\{x\in \R\;|\;x>0\}$, for any 
$Y\subset_{fin}G\setminus N$, for any $\Phi\subset_{fin}N$, and for any $\epsilon>0$ 
there exists a homomorphism
$\phi:G\to H$, $H\in\cK$, and $\|\cdot\|\in\cL_H$ such that
\begin{itemize}
\item $\|\phi(y)\|>\alpha_y$ for $y\in Y$,
\item $\|\phi(x)\|<\epsilon$ for $x\in\Phi$.
\end{itemize}
\end{definition}
The following proposition is a modification of Proposition~1.7 in \cite{Thom1}, see also
Lemma~6.2 of \cite{GR}. 
\begin{proposition}\label{prop_esen}
If $N$ is a $(\cK,\cL)$-separated normal subgroup of $G$
then $G/N$ is $(\cK,\cL)$-approximable.

Let $F$ be a free group and  $N\triangleleft F$. If $F/N$ is $(\cK,\cL)$-approximable 
then $N$ is a $(\cK,\cL)$-separated normal subgroup of $F$.
\end{proposition}
{\bf Remark.} The difference of Proposition~\ref{prop_esen} with Proposition~1.7 of 
\cite{Thom1} is 
that in the last one the group $G$ is assumed to be free. This assumption is 
not required for the first implication. In the second part of Proposition~\ref{prop_esen}
the freeness of $F$ is used to extend a map from generators of $F$ to a homomorphism 
required by Definition~\ref{def_metric_separ}.  

\subsection{$\cK$-approximation}
In this subsection we define the notion of $\cK$-approximation  which is 
independent of the length functions.  In what follows we use the notation $x^g=g^{-1}xg$.
\begin{definition}
Let $G$ be a group and $X\subseteq G$. Let 
$C_n(X,G)=\{x_1^{g_1}x_2^{g_2}\dots x_n^{g_n}\;|\;x_i^{\pm}\in X,\,g_i\in G\}$ (the set of
$n$-consequences of $X$ in $G$). We just write $C_n(X)$ if the group $G$ is uniquely 
assumed by a context.  

Let $X,Y\subset G$. We say that $Y$ is $n$-separated from $X$ (in $G$) if
$Y\cap C_n(X,G)=\emptyset$.
\end{definition}  
\begin{definition}\label{def_w-approx}
Let $\cK$ be a class of groups. A group $G$ is $\cK$-approximable
if for any $\Phi\subset_{fin}G$ for any $n\in\N$ there exist $H\in\cK$ and a map 
$\phi:\Phi\to H$ such that 
\begin{itemize}
\item $\phi(1)=1$;
\item the set $\phi(\Phi\setminus\{1\})$ is $n$-separated from 
$\{\phi(g)\phi(h)(\phi(gh))^{-1}\;|\;g,h,gh\in\Phi\}$.
\end{itemize}
Let $\cK.approx$ denote the class of $\cK$-approximable groups.
\end{definition}
If $G$ is $\cK$-approximable (resp. $(\cK,\cL)$-approximable)
then any subgroup of $G$ is  $\cK$-approximable (resp. $(\cK,\cL)$-approximable).
It is easy to  verify and we will often use it without mention explicitly.    
There is the following relation 
between $(\cK,\cL)$-approximation and $\cK$-approximation. 
\begin{lemma} \label{lm_w_and_m}
If a group $G$ 
is $(\cK,\cL)$-approximable then it
is $\cK$-approximable; if $G$ is $\cK$-approximable 
then one can define 
invariant length functions $\cL$ such that $G$ is $(\cK,\cL)$-approximable.
\end{lemma}
\begin{proof}
Fix a group $H$ with an invariant length function $\|\cdot\|$, $\epsilon>0$ and $n\in\N$. Let
$X=\{h\in H\;|\; \|h\|<\epsilon\}$ and $Y=\{h\in H\;|\;\|h\|\geq n\epsilon\}$.
Then $Y$ is $n$-separated from $X$. Comparing Definition~\ref{def_m-approx} and
Definition~\ref{def_w-approx} gives us the first part of the Lemma.

Let $X,Y\subset H$ and $Y$ be $n$-separated from $X$. Denote by 
$X^H=\{x^h\;|\;x^{\pm 1}\in X,\,h\in H\}$. The Cayley graph $\Gamma(H,X^H)$ defines a 
distance
on $H$, $d(h_1,h_2)=$'the length of the shortest paths from $h_1,h_2$ in $\Gamma$ or $\infty$ if there is no such a path'.
It is easy to check that $\|h\|=\min\{\frac{1}{n}d(1,h),1\}$ defines an invariant length function
on $H$ such that $\|x\|=1/n$ for $x\in X$ and $\|y\|=1$ for $y\in Y$.  Now, one may
convert a $\cK$-approximation into a $(\cK,\cL)$-approximation  where $\cL$ consists of
the above constructed length functions.
\end{proof}
We give  analogues of Definition~\ref{def_metric_separ} and  
Proposition~\ref{prop_esen}.
\begin{definition}\label{def_separ}
Let $N\triangleleft G$. We say that $N$ is $\cK$-separated normal subgroup of 
$G$ if for any $n\in \N$, for any $Y\subset_{fin}G\setminus N$, and for
any $\Phi\subset_{fin}N$ there exists a homomorphism $\phi:G\to H$, $H\in\cK$,
such that $\phi(Y)$ is $n$-separated from $\phi(\Phi)$.
\end{definition}
\begin{proposition}\label{prop_weak_esen}
If $N$ is a $\cK$-separated normal subgroup of $G$, then $G/N$ is $\cK$-approximable. 

Let $N\triangleleft F$ for a free group $F$. If $F/N$ is $\cK$-approximable then $N$ is a
$\cK$-separated normal subgroup of $F$.
\end{proposition}

Let $G=\prod_{i\in\N} H_i$, be an (unrestricted) direct product. 
For  $i\leq j$ let
$\Pr_i^j:G\to H_i\times H_{i+1}\times\dots\times H_j$ be the natural projection. Let
$\Pr_i=\Pr_i^i$. We equip $G$ with the product (or Tychonoff) topology. 
(In the most applications $H_i$ are finite groups equipped with discrete topology.)
For a class $\cK$ let $\prod(\cK)$ contains $\cK$ and all finite direct products
of groups from $\cK$. 
\begin{proposition}\label{prop_direct}
Suppose that $\cK$ is a class of compact groups.
Let $X$ be a closed subgroup of $G=\prod_{i\in\N}H_i$, $H_i\in\cK$, 
such that $\Pr_i(X)\in\cK$ for any $i\in\N$. Let $N\triangleleft X$. 
Then $X/N$ is $\prod(\cK)$-approximable. (Note, that $N$ need not to be topologically closed in G).  
\end{proposition}  
\begin{proof}
Let $\Phi\subset_{fin} X$ and $n\in \N$. Let $\Phi_N=\Phi\cap N$ and 
$\Phi_0=\Phi\setminus N$. 
Notice that $C_n(\Phi_N,X)\cap\Phi_0\subset N\cap \Phi_0=\emptyset$. 
But 
$$
C_n(\Phi_N,X)=\bigcup_{\bar f \in \Phi_N^n} \{f_1^{x_1}f_2^{x_2}\dots f_n^{x_n}\;|\;x_i\in X\}
$$ 
is a finite union of continuous  images 
of a compact set $X^n$, hence $C_n(\Phi_N,X)$ is a compact set. So, there exists an open neighborhood $U$ of $\Phi_0$ such that
$C_n(\Phi_N,X)\cap U=\emptyset$. It follows that 
$$
C_n(\Pr\nolimits_1^k(\Phi_N),\Pr\nolimits_1^k(X))\cap\Pr\nolimits_1^k(\Phi_0)=\Pr\nolimits_1^k(C_n(\Phi_N,X))\cap \Pr\nolimits_1^k(\Phi_0)=\emptyset
$$ 
for some $k>0$ (the pre-images of  open sets 
with respect to $\Pr_1^k$ form the base of the Tychonoff topology). So, the
homomorphisms $\Pr_1^k:X\to \Pr_1^k(X)$ satisfy Proposition~\ref{prop_weak_esen} for
the class $\prod(\cK)$. 
\end{proof}
\begin{corollary}\label{cor_prod1}
Let $\cK$ be a class compact groups closed with respect of taking subgroups and finite direct products.  
Let $X$ be a closed subgroup of a direct product of groups from $\cK$.
Then any quotient of $X$ is approximable by $\cK$.  
Particularly, it is true for  $\cK=Nil,\,Sol$, $Fin$, or class of all compact groups.
\end{corollary}
\begin{corollary}\label{cor_sof}
Let $\cK=Alt$ or $\cK=Sym$.
Then any quotient of a direct product $\prod_{i\in\N}H_i$, $H_i\in\cK$ 
is $\cK$-approximable.  
\end{corollary}
\begin{proof}
In Subsection~\ref{subsec_sof} and Subsection~\ref{sub_proof} we show that the classes of $Sym$, $\prod(Sym)$,
$Alt$, $\prod(Alt)$-approximable groups coincide with the class of sofic groups.
\end{proof}

\subsection{Sofic groups}\label{subsec_sof}
Classically, sofic groups are defined using metric approximations. In this subsection
we show that we can avoid the use of the length functions. 
Let $S_m$ denote the group of all permutations of $[m]=\{1,2,...,m\}$ 
(symmetric group on $m$ elements);
$A_m\triangleleft S_m$ denote alternating group on $m$ elements.
The normalized Hamming length function $\|\cdot\|$ is, by definition,
$$
\|h\|=\frac{|\{x\in [m]\;|\;xh\neq x\}|}{m},
$$ 
here we suppose that  $h\in S_m$ has a right natural action, $h:x\to xh$, on $[m]$.
It is defined on $S_m$ as well as on $A_m$ for any $m\in\Z^+$.
Let $\cH$ denote the class of all normalized Hamming length functions.
\begin{definition}\label{def_sof}
$(Sym,\cH)$-approximable groups are said to be sofic groups.
\end{definition}

\begin{lemma}\label{lm_sym_alt}
$(Alt,\cH).approx$ coincides with the class of all sofic groups.
\end{lemma}
\begin{proof}
$A_m<S_m$. So any $(Alt,\cH)$-approximation is a $(Sym,\cH)$-approximation. 
It proves that $(Alt,\cH).approx\subseteq (Sym,\cH).approx$. On the
other hand, $S_m\hookrightarrow A_{2m}$ (just repeating each cycle of a permutation twice).
This inclusion preserve $\|\cdot\|$. So, any $(Sym,\cH)$-approximation may 
be converted into
an $(Alt,\cH)$-approximation. It shows that 
$$
(Sym,\cH).approx\subseteq (Alt,\cH).approx.
$$
\end{proof}
\begin{proposition}\label{sofic_without_metric}
$Sym.approx=\prod(Sym).approx=Alt.approx=\prod(Alt).approx=$\\''sofic groups''.
\end{proposition}

We will proof the proposition in subsection~\ref{sub_proof}. 
Now we demonstrate some auxiliary results. The following lemma is an analogue
of Lemma~2.5 of \cite{EL} with a similar proof.
 
\begin{lemma}\label{lm_brenner}
Let $A$ be an alternating group, $X\subset A$ and $y\not\in C_n(X,A)$. Then
$\frac{\|y\|}{\|x\|}\geq \frac{n-1}{16}$ for any $1\neq x\in X$. In other words,
denoting $\epsilon=\sup\{\|x\|\;|\;x\in X\}$ one gets 
$\{\alpha\in A\;|\;\|\alpha\|<\frac{n-1}{16}\epsilon\}\subset C_n(X,A)$.
\end{lemma}
\begin{proof}
This lemma is a manifestation of the fact that in a finite simple group powers of
a conjugacy class cover the group ``almost as fast as possible'' \cite{Shalev}.
The case of alternating groups was considered in \cite{Brenner}.
Let $C_x$ denote the conjugacy class of $x\neq 1$ in $A$.
Lemmas~2.05, 2.06, 3.03 of \cite{Brenner} imply that $C_x^4$ contains  all nontrivial
even permutations of support of $x$. (Support of $x$ are elements which are not
fixed by $x$.)  Then we may shift support by conjugation and construct 
$r=\left\lceil \frac{\|y\|}{\|x\|}\right\rceil$ 
permutations whith supports partitioning the support of $y$. So, $C_x^{4r}$ contains 
a permutation with the same support
as $y$. Applying ones again Lemmas~2.05, 2.06, 3.03 of \cite{Brenner}  we obtain that 
$y\in C_x^{16r}$.
\end{proof}
Remark. It looks like that one may change $\frac{n-1}{16}$ by $\frac{n-1}{4}$ 
in the lemma. For this we should start with $r$ shifts of $x$ and cover the support of 
$y$, but we need 
some extra details if the support of $y$ is almost all set $[m]$... 
\begin{lemma}\label{lm_ampl}
Let $G$ be a group, $X\subset G$, and $y\in G$. Suppose that there exists a
homomorphism $\phi:G\to A_m$, such that $\|\phi(y)\|/\|\phi(x)\|>n\geq 2$ for
any $x\in X$. Then there exists $r=r(y)$ and a homomorphism $\psi:G\to A_{mr}$, such that 
$\|\psi(y)\|\geq 1/2$ and $\|\psi(x)\|<1-(1/4)^{1/n}$. 
\end{lemma}
\begin{proof}
If $\|\phi(y)\|\geq 1/2$, put $r=1$ and $\psi=\phi$. We are done.  
If $\|\phi(y)\|<1/2$ the arguments are based on the amplification trick, see \cite{EL}. 
To make exposition self-contained we explain it here. Define an
inclusion $h\to h^{\otimes r}:A_m\to A_{mr}$ as follows.
Consider $A_{mr}$ as even permutations of $[m]^r$. For $(j_1,j_2,\dots,j_r)\in [m]^r$ let
$(j_1,j_2,\dots,j_r)h^{\otimes r}=(j_1h,j_2h,\dots,j_rh)$. It is easy to check that 
$1-\|h^{\otimes r}\|=(1-\|h\|)^r$.
 
Now, let $\psi(z)=\phi^{\otimes r}(z)$, then
$1-\|\psi(z)\|=(1-\|\phi(z)\|)^r$. We may choose $r\in \N$ such that 
$1/4<1-\|\psi(y)\|\leq 1/2$. We just need to estimate $\psi(x)$ for $x\in X$:
$$
\frac{\log(1-\|\psi(y)\|)}{\log(1-\|\psi(x)\|)}=
\frac{\log(1-\|\phi(y)\|)}{\log(1-\|\phi(x)\|)}\geq \frac{\|\phi(y)\|}{\|\phi(x)\|}>n.
$$   
So, the estimate for $\|\psi(x)\|$ follows.
\end{proof}
The following lemma is a strengthening of Proposition~\ref{prop_esen} 
for sofic groups.
\begin{lemma}\label{lm_sof_esen}
Let $G$ be a group and $N\triangleleft G$. Suppose, that for any $y\in G\setminus N$,
any $\Phi\subset_{fin}N$, and any $\epsilon>0$, there exists a homomorphism
$\phi:G\to A_m$ such that
$\|\phi(y)\|\geq 1/2$ and $\|\phi(x)\|<\epsilon$ for any $x\in \Phi$. 
Then $G/N$ is sofic.  
\end{lemma}
\begin{proof}
\begin{claim}\label{cl_1}
Let $G$ be a group, $X\subset G$, $Y\subset_{fin} G$, and $\epsilon>0$. 
Suppose, that for any $y\in Y$ there exists
a homomorphism $\phi_y:G\to A_{m_y}$ such that $\|\phi_y(y)\|\geq 1/2$ and
$\|\phi_y(x)\|< \epsilon$ for any $x\in X$. Then there exists 
$\psi:G\to A_m$ such that $\|\psi(y)\|\geq\frac{1}{2|Y|}$ for any $y\in Y$ and 
$\|\psi(x)\|<\epsilon$ for any $x\in X$.
\end{claim}
\begin{proof}
For $\alpha\in A_r$ and $\beta\in A_k$ we define $\alpha\oplus\beta\in A_{r+k}$. Consider
$A_{r+k}$ as permutation groups on disjoint union $[r]\dot\cup [k]$. Then
$$
x(\alpha\oplus\beta)=\left\{\begin{array}{lll} x\alpha &\mbox{if} & x\in [r]\\
                                               x\beta &\mbox{if} & x\in [k]
                     \end{array}\right.   
$$
It is clear that $\|\alpha\oplus\beta\|=\frac{r\|\alpha\|+k\|\beta\|}{r+k}$.
Let $r=\gcm(\{m_y\;|\;y\in Y\})$. Replacing each $\phi_y$ by the sum 
$\phi_y\oplus\phi_y\dots$ ($r/m_y$ copies of $\phi_y$)
we may assume that $\phi_y:G\to A_r$ for every $y\in Y$. Now, take 
$\psi(z)=\bigoplus\limits_{y\in Y}\phi_y(z)$.
\end{proof}
In order to prove the lemma we use Proposition~\ref{prop_esen}.
Fix $Y\subset_{fin}G\setminus N$, $\Phi\subset_{fin}N$, $\epsilon>0$ we show that 
the hypothesis of the Lemma 
guarantee the assumptions of  Proposition~\ref{prop_esen} for $\alpha_y=1/3$. 
Choose sufficiently small $0<\epsilon'$. 
For every $y\in Y$ consider $\phi_y:G\to A$ such that $\|\phi_y(y)\|\geq 1/2$ 
and $\|\phi_y(x)\|<\epsilon'$ for $x\in\Phi$. Then construct $\psi$ of 
Claim~\ref{cl_1}. Applying the amplification trick to $\psi$ we obtain
$\phi$ of Proposition~\ref{prop_esen}. Notice, that the choice of
$\epsilon'$ depends on $\epsilon$ and $|Y|$ only. 
\end{proof}

\subsection{Proof of Proposition~\ref{sofic_without_metric}}\label{sub_proof}
Let us start with $Alt$ parts of the proposition.
Notice, that 'sofic groups''\\ $=(Alt,\cH).approx\subseteq Alt.approx \subseteq\prod(Alt).approx$  by Lemma~\ref{lm_w_and_m} and Lemma~\ref{lm_sym_alt}. 
So, it suffices to show that 
$\prod(Alt)$-approximable groups are 
$(Alt,\cH)$-approximable. Let $G=F/N$ be $\prod(Alt)$-approximable, $F$ be a free group.
Let $y\in F\setminus N$ and $\Phi\subset_{fin}N$. Fix $\epsilon>0$ and $n\in\N$ such that
$\epsilon<1-(1/4)^\frac{16}{n-1}$.
By Proposition~\ref{prop_weak_esen} chose a homomorphism 
$\phi:F\to \tilde A=\prod\limits_{j=1}^kA_j$ such that $\{\phi(y)\}$ is $n$-separated from 
$\phi(\Phi)$.  Let $\epsilon_j=\max\{\|\Pr_j(\phi(x))\|\;|\;x\in\Phi\}$ and 
$D_j=\{y\in A_j\;|\;\|y\|<\frac{n-1}{16}\epsilon_j\}$. 
As $Pr_j(C_n(X,\tilde A))=C_n(Pr_j(X),A_j)$, 
by Lemma~\ref{lm_brenner} we get
$\prod\limits_{j=1}^kD_j\subset C_n(\phi(\Phi),\tilde A)$. So, $Pr_j\phi(y)\not\in D_j$ for some $j$.
The $Alt$-part of the proposition follows by Lemmas~\ref{lm_ampl} and \ref{lm_sof_esen}.

Let $G$ be a $\prod(Sym)$-approximable group. By Proposition~\ref{prop_ultr} and 
Lemma~\ref{lm_w_and_m} $G$ is a subgroup of $P$, a quotient of a direct product of
symmetric groups.  Symmetric groups are  extensions of alternating groups by the cyclic
group with 2 elements.
So, $P$ is an extension of $Q$, a quotient of a direct product
of alternating groups, by an abelian group. 
As the proposition and Corollary~\ref{cor_sof} is proven for 
$Alt$, the group $Q$ is sofic. Now, $P$ is sofic as an extension of $Q$ by amenable 
group, see \cite{EL}. 
\section{$\Gamma$-groups.}\label{sec3}
\begin{definition}\label{def_Fgroup}
Let $\Gamma$ be a group. A group $G$ with injective homomorphism $\Gamma\hookrightarrow G$ 
is said to be an $\Gamma$-group. 
One may think that $G$ has a fixed copy of $\Gamma <G$. Let $G_1$, $G_2$  
be $\Gamma$-groups. A homomorphism $\phi:G_1\to G_2$
is called a $\Gamma$-morphism if the restriction of $\phi$ on $\Gamma$ 
is the identity map, $\phi|_\Gamma=Id_\Gamma$. 
Notation $\phi:G_1\to_\Gamma G_2$ means $\phi$ is a 
$\Gamma$-morphism. 
\end{definition}
This definition has an important use in algebraic geometry over groups, \cite{Baumslag1}.
Clearly,  $\Gamma$-groups form a category.
For an $\Gamma$-group $G$ and $N\triangleleft \Gamma$ we denote 
$\bar{N}^G=\Gamma\cap \ln{N}_G$. 
It is clear that $\bar{\cdot}^G$ 
is an idempotent.
The condition $N=\bar N^G$ is equivalent 
to the existence of the following commutative diagram with injective $\psi$:
\begin{equation}\label{diagram1}
\xymatrix{\Gamma \ar@{^{(}->}^{\phi}[r]\ar@{->}[d] & G\ar@{->}[d]\\
         \Gamma /N\ar@{^{(}->}^{\psi}[r] & H
         }  
\end{equation} 
In the category of $\Gamma$-groups  not for all pairs of objects $G_1$, $G_2$ there 
exists a morphism $G_1\to_\Gamma G_2$. (For example, $G_1$ is a simple group, $\Gamma$ is 
a proper subgroup of $G_1$ and $G_2=\Gamma$.) 
\begin{lemma} \label{lm_prec}
Suppose that  there exists $\Gamma$-morphism $H\to_\Gamma G$. 
Then $\bar{N}^H\subseteq \bar{N}^G$ for any $N\triangleleft \Gamma$.
\end{lemma}
\begin{proof}
Let $\phi:H\to_\Gamma G$ be a $\Gamma$-morphism. Then 
$$
\bar N^G=\ln{N}_G\cap \Gamma\supseteq \ln{\phi(N)}_{\phi(H)}\cap\phi(\Gamma)=\\
\phi(\ln{N}_H)\cap\phi(\Gamma)\supseteq \phi(\ln{N}_H\cap \Gamma)=\bar N^H
$$
\end{proof}
Let $G$ be a $\Gamma$-group. Denote 
$$
\FP(G)=\{H\;:\;H\mbox{ is a finitely presented $\Gamma$-group and } 
\exists \phi: H\to_\Gamma G\},
$$ 
here we call a $\Gamma$-group $H$ finitely presented if it is finitely presented 
relative to $\Gamma$, that is, $H=(\Gamma*F)/M$, where $F$ is a free group of finite rank
and $M\triangleleft (Gamma*F)$ is finitely generated (as a normal subgroup).   
\begin{lemma}\label{lm_fp}
Let $G$ be a $\Gamma$-group and $N\triangleleft \Gamma$. Then 
$$
\bar{N}^G=\bigcup_{H\in\FP(G)} \bar{N}^H
$$ 
\end{lemma}
\begin{proof}
Lemma~\ref{lm_prec} shows that $\bar{N}^G$ contains $\bigcup\limits_{H\in\FP(G)} \bar{N}^H$. 
Now, suppose, that $w\in \bar{N}^G$. Consider $G=\Gamma*F_\infty/M$, where 
$F_\infty$ is a free group of countable rank. 
Then $w$ may be presented in $\Gamma*F_\infty$ as a finite product $w=w_1w_2...w_r$, where
each $w_i$ either in $M$ or equal to a conjugate (in $\Gamma*F_\infty$) 
of some element of $N$.
Let $F_k<F_\infty$ be a free group generated by all elements of $F_\infty$ appearing in 
$w_i$, $i=1,\dots,r$. Let $\tilde M\triangleleft \Gamma*F_k$ be generated by
$\{w_1,\dots, w_r\}\cap M$. Notice, that $H=(\Gamma*F_k)/\tilde M$ is a $\Gamma$-group 
(as it has less relations then $G$). It follows that
$H\in\FP(G)$ and $w\in \bar{N}^H$.
\end{proof}

Recall, that $\ls{\bar a}$ denotes the free group generated by $\bar a=a_1,\dots,a_r$.
Let $\cK$ be a class of finite groups. 
Consider the set of all homomorphisms $\phi:\ls{\bar a}\to H$, for all $H\in\cK$.
This set is countable and we may enumerate it, say $\phi_i:\ls{\bar a}\to H_i$. 
Let $\cK(\ls{\bar a})=\prod\limits_{i=1}^\infty H_i$. 
Then $\phi_i$ define homomorphism $\phi_\infty:\ls{\bar a}\to \cK(\ls{\bar a})$, 
$\phi_\infty(w)=(\phi_1(w),\phi_2(w),\dots,\phi_j(w),\dots)$.
A free group $\ls{\bar a}$ is residually $\cK$ if $\phi_\infty$ is an injection. 
Let us denote by $\cK'(\ls{\bar a})$ the closure of 
$\phi_\infty(\ls{\bar a})$ with respect to the product
topology.

\begin{theorem}\label{th_approx1}
Let $\cK\subseteq Fin$ such that $\cK.approx=\prod(\cK).approx$. 
Let a free group $\ls{\bar a}$ be residually $\cK$.
Consider $\cK(\ls{\bar a})$ as an $\ls{\bar a}$-group. 
Then $\ls{\bar a}/N$ is $\cK$-approximable  if and only if $N=\bar N^{\cK(\ls{\bar a})}$.

Suppose, in addition, that $\cK$ is closed under taking subgroups. 
Then $\ls{\bar a}/N$ is $\cK$-approximable  
if and only if $N=\bar N^{\cK'(\ls{\bar a})}$.
\end{theorem}
\begin{proof}
We start with the first part of the theorem.

$\Longrightarrow$. Let $\ls{\bar a}/N$ be $\cK$-approximable. Suppose, searching for 
a contradiction,
that there exists $y\in\ls{\bar a}\setminus N$ such that $y\in \ln{N}_{\cK(\ls{\bar a})}$.
It follows that $y\in C_n(X,\cK(\ls{\bar a}))$ for some $n\in\N$ and some 
$X\subset_{fin} N$. Using definition of $\cK(\ls{\bar{a}})$ and applying projection on $H_i$ we get $\phi_i(y)\in C_n(\phi_i(X),H_i)$.
A contradiction with 
Proposition~\ref{prop_weak_esen} as $(H_i,\phi_i)$ run over all pairs 
$H\in\cK$, $\phi:\ls{\bar a}\to H$. 

$\Longleftarrow$. Let $N=\bar N^G$. By Diagram~(\ref{diagram1}) we see that 
$\ls{\bar a}/N$ is a subgroup of $Q=G/\ln{N}_G$. 
$Q$ is $\prod(\cK)$-approximable by Proposition~\ref{prop_direct}.
We are done by the hypothesis $\cK.approx=\prod(\cK).approx$ and the fact that subgroups of $\cK$-approximable groups are $\cK$-approximable.

The second part of the theorem.  Let $G=\cK(\ls{\bar a})$ and $H=\cK'(\ls{\bar a})$. 
As $H\hookrightarrow_{\ls{\bar a}}G$ one has
$\bar N^H\subseteq \bar N^G$ by Lemma~\ref{lm_prec}. So, if $\ls{\bar a}/N$ is 
$\cK$-approximable then $N\subseteq \bar N^H\subseteq \bar N^G=N$ by the first part of the theorem.
On the other hand, if $\bar N^H=N$ then $\ls{\bar a}/N\hookrightarrow H/\ln{N}_H$.
By Proposition~\ref{prop_direct} $H/\ln{N}_H$ is $\prod(\cK)$-approximable as well as all of its
subgroups. We are done by the hypothesis $\cK.approx=\prod(\cK).approx$.

{\bf Remark} In fact, it is possible to show that $\bar N^H=\bar N^G$. Indeed, under assumption of the theorem
there exists a $\ls{\bar a}$-morphism $G\to_{\ls{\bar a}} H$.
\end{proof}
\section{Soficity and equations over groups}\label{sec_equation}

Let us consider the concepts given in Definition~\ref{def_sys} of the introduction. A system from $Sys(\cK)$ is
said to be a $\cK$-system.  
It is clear, that if a system is solvable in $G$ then it is solvable in any quotient of $G$.
Also, any $\cK$-system is solvable in any direct product of
 groups
from $\cK$. Using Proposition~\ref{prop_ultr}  and Lemma~\ref{lm_w_and_m} we obtain the following proposition,
see \cite{Pestov1}.
\begin{proposition}\label{prop_sys1}
If $G$ is $\cK$-approximable  then any $\cK$-system is solvable over $G$.
\end{proposition}
The following lemma is trivial.
\begin{lemma}\label{lm_sys1}
$\ls{a_1,\dots,a_r,x_1,\dots,x_n\;|\;\bar{w}(\bar{a},\bar{x}}\in \FP(\cK(\ls{\bar a}))$
if and only if $\bar w\in Sys(\cK)$
\end{lemma}
This lemma, Theorem~\ref{th_approx1}, and Proposition~\ref{prop_sys1} yield the following corollary.
\begin{corollary}\label{cor_sys1}
Let $\cK\subseteq Fin$ and $\cK.approx=\prod(\cK).approx$. Then a group $G$ is 
$\cK$-approximable if and only
if any $\cK$-system is solvable over $G$.
\end{corollary}

\section{Algebraic groups}\label{sec_algebraic_groups}

Here we study algebraic groups over algebraically closed fields (AGCF).

\begin{proposition} \label{prop_Fin->AGCF}
Any $\bar w\in Sys(Fin)$ is solvable in an AGCF group. 
\end{proposition} 
\begin{proof}
The proof is very similar to the proof of Malcev theorem in \cite{Brown1}, 
see  also \cite{Arzhan1}. 

Let $G$ be an algebraic group over algebraically closed field $K$.
Let $\cD\subset K[\bar z]$ be the defining equations $\bar g\in G^r$.
Suppose, that $\bar w(\bar g,\bar x)$ is a system without solution in $G$. 
Let $\cS\subset K[\bar z]$ be the algebraic equations defined by $\bar w$.
By Hilbert's Nullstellensatz $1\in I(\cS)$, where $I(X)$ is an ideal generated by $X$. 
This means that  
there exist $f_1,f_2,..,f_k\in K[\bar{z}]$, and $h_1,h_2,...,h_k \in \cS$ such that 
\begin{equation} \label{eq_id}
1=\sum_{i=1}^k g_ih_i.
\end{equation}
Let $B$ be the set of elements of $K$ appearing in  $g_i$ , $h_i$, $\cD$, $\cS$. 
The set $B$ is finite. Consider the ring $R\subset K$, generated by $B$. Clearly,
Eq.(\ref{eq_id}) holds in $R[\bar{z}]$ and no nontrivial homomorphism $R\to R/I$ sends $1$ 
to $0$. 
Now, there exists
a maximal ideal $I$ of $R$. The corresponding field $R/I$ is finite \cite{Brown1}. 
Let $\phi:R\to R/I$ be the natural morphism. This morphism define a group homomorphism $\phi:G\to\phi(G)$.  
Then the system
$\phi(\cS)$ does not have solutions in $R/I$ and $\bar w(\phi(\bar g),x)$ does not
have solutions in the finite algebraic group $\phi(G)$.  
\end{proof}

\begin{corollary}
Let $A$ be a quotient of a direct product of  AGCF groups. Suppose that $G<A$. 
Then the group $G$ is $Fin$-approximable, that is, $G$ is weakly sofic. 
\end{corollary}   
\begin{corollary}
Let $\cK$ be a class of AGCF groups. Then $\cK.approx\subset Fin.approx$.
\end{corollary}
Notice, that the proof of Proposition~\ref{prop_Fin->AGCF} essentially uses the algebraic
closeness. For example, $x^2+y^2=-1$ has no solutions in $\R$, but has a solution in any
finite field. The hyperlinear groups are defined as being approximable by finite dimensional unitary groups 
(with normalized Hilbert-Schmidt norm as a length function) \cite{Lupini1, Pestov1}. 
So, the following is an open question.
\begin{question}
Are hyperlinear groups  weakly sofic?  
\end{question}

\section{Proof of Proposition (Some equivalence).}
\begin{description}
\item{(1)$\Longrightarrow$(2).} This is an implication of  Theorem~\ref{th_approx1}.
\item{(2)$\Longrightarrow$(3).} This is trivial.
\item{(3)$\Longrightarrow$(4).}  Given a finite set $\cF$, using small cancellation theory we may find
$H<F$ such that $\ln{H}_F$ avoids $\cF\not\ni 1$ and $H$ has CEP in $F$.
It follows that $\hat F$ has a free subgroup with CEP. So, $\hat F$ is SQ-universal. 
Explicitly, let $A_i=a^{100}b^ia^{101}b^i\dots a^{199}b^i$. Then $H=\ls{A_i,A_{i+1}}$
satisfies our properties for $i$ sufficiently large by results of \cite{Olshanski}.

\item{(4)$\Longrightarrow$(1).} As $\bar w\in Sys(\cK)$ has a solution in any quotient of 
$\hat F$, the system $\bar w$ is solvable over any (countable) group $G$. And we are done by
Corollary~\ref{cor_sys1}.   
\end{description}

{\bf Acknowledgments.}
Some ideas of the work appeared during the ``Word maps and stability of representation'' 
conference
and the author's visit at the  ESI, Vienna, April 2013. 
The author thanks the organizers and participants of the conference,
especially Goulnara Arzhantseva, Jakub Gismatullin, and 
Kate Juschenko for useful discussions.
The author thanks Andreas Thom for pointing out an error in the previous version.
The visit was supported by CRSI22-130435 grant of SNSF.      
Also the work was partially supported by PROMEP grant UASLP-CA21 and 
by grant 14-41-00044 of RSF at the Lobachevsky University of Nizhny Novgorod".


\begin{thebibliography}{10}

\bibitem{AGG}
M.~A. Alekseev, L.~Yu. Glebski{\u\i}, and E.~I. Gordon.
\newblock On approximations of groups, group actions and {H}opf algebras.
\newblock {\em Zap. Nauchn. Sem. S.-Peterburg. Otdel. Mat. Inst. Steklov.
  (POMI)}, 256(Teor. Predst. Din. Sist. Komb. i Algoritm. Metody. 3):224--262,
  268, 1999.

\bibitem{Arzhan2}
G.~N. Arzhantseva.
\newblock Asymptotic approximaions of finitely generated groups.
\newblock In Guillamon~Antoni Ventura, Enric, editor, {\em Research
  Perspectives CRM Barcelona-Fall}, volume~1 of {\em Trends in Mathematics},
  pages 7--16. Springer, 2014.

\bibitem{Arzhan1}
G.~N. Arzhantseva and L.~Paunescu.
\newblock Linear sofic groups and algebras.
\newblock accepted in Transactions of the American Mathematical Society,
  available at http://www.ams.org/cgi-bin/mstrack/accepted\_papers/tran, 2012.

\bibitem{Gis}
Goulnara Arzhantseva and Jacub Gismatullin.
\newblock Personal communication.

\bibitem{Baumslag1}
Gilbert Baumslag, Alexei Myasnikov, and Vladimir Remeslennikov.
\newblock Algebraic geometry over groups. {I}. {A}lgebraic sets and ideal
  theory.
\newblock {\em J. Algebra}, 219(1):16--79, 1999.

\bibitem{Brenner}
J.~L. Brenner.
\newblock Covering theorems for {FINASIG}s. {VIII}. {A}lmost all conjugacy
  classes in {${\cal A}_{n}$} have exponent {$\leq {}4$}.
\newblock {\em J. Austral. Math. Soc. Ser. A}, 25(2):210--214, 1978.

\bibitem{Brick1}
Stephen~G. Brick.
\newblock Normal-convexity and equations over groups.
\newblock {\em Invent. Math.}, 94(1):81--104, 1988.

\bibitem{Brick2}
Stephen~G. Brick.
\newblock Relative normal-convexity and amalgamations.
\newblock {\em Bull. Austral. Math. Soc.}, 44(1):95--107, 1991.

\bibitem{Brown1}
Nathanial~P. Brown and Narutaka Ozawa.
\newblock {\em {$C^*$}-algebras and finite-dimensional approximations},
  volume~88 of {\em Graduate Studies in Mathematics}.
\newblock American Mathematical Society, Providence, RI, 2008.

\bibitem{Lupini1}
Valerio Capraro and Martino Lupini.
\newblock {\em Introduction to {S}ofic and hyperlinear groups and {C}onnes'
  embedding conjecture}, volume 2136 of {\em Lecture Notes in Mathematics}.
\newblock Springer, Cham, 2015.
\newblock With an appendix by Vladimir Pestov.

\bibitem{Coulbois}
Thierry Coulbois and Anatole Khelif.
\newblock Equations in free groups are not finitely approximable.
\newblock {\em Proc. Amer. Math. Soc.}, 127(4):963--965, 1999.

\bibitem{EL}
G{\'a}bor Elek and Endre Szab{\'o}.
\newblock Hyperlinearity, essentially free actions and {$L^2$}-invariants.
  {T}he sofic property.
\newblock {\em Math. Ann.}, 332(2):421--441, 2005.

\bibitem{GR}
Lev Glebsky and Luis~Manuel Rivera.
\newblock Sofic groups and profinite topology on free groups.
\newblock {\em J. Algebra}, 320(9):3512--3518, 2008.

\bibitem{Gordon}
E.~Gordon.
\newblock Persanal communication.

\bibitem{Gr99}
M.~Gromov.
\newblock Endomorphisms of symbolic algebraic varieties.
\newblock {\em J. Eur. Math. Soc. (JEMS)}, 1(2):109--197, 1999.

\bibitem{Higman}
Graham Higman and Elizabeth Scott.
\newblock {\em Existentially closed groups}, volume~3 of {\em London
  Mathematical Society Monographs. New Series}.
\newblock The Clarendon Press, Oxford University Press, New York, 1988.
\newblock Oxford Science Publications.

\bibitem{Howie}
James Howie.
\newblock The {$p$}-adic topology on a free group: a counterexample.
\newblock {\em Math. Z.}, 187(1):25--27, 1984.

\bibitem{Khar_Myas}
Olga Kharlampovich and Alexei Myasnikov.
\newblock Elementary theory of free non-abelian groups.
\newblock {\em J. Algebra}, 302(2):451--552, 2006.

\bibitem{Levin}
Frank Levin.
\newblock Solutions of equations over groups.
\newblock {\em Bull. Amer. Math. Soc.}, 68:603--604, 1962.

\bibitem{Shalev}
Martin~W. Liebeck and Aner Shalev.
\newblock Simple groups, permutation groups, and probability.
\newblock {\em J. Amer. Math. Soc.}, 12(2):497--520, 1999.

\bibitem{Lyndon_Schupp}
Roger~C. Lyndon and Paul~E. Schupp.
\newblock {\em Combinatorial group theory}.
\newblock Classics in Mathematics. Springer-Verlag, Berlin, 2001.
\newblock Reprint of the 1977 edition.

\bibitem{NN}
B.~H. Neumann and Hanna Neumann.
\newblock Embedding theorems for groups.
\newblock {\em J. London Math. Soc.}, 34:465--479, 1959.

\bibitem{Olshanski} 
 Aleksandr Yu.  Ol'shanskii.
\newblock The SQ-universality of hyperbolic groups.
\newblock {\em Sbornik: Mathematics}, 186(8):1199--1211, 1995.

\bibitem{Pestov1}
Vladimir~G. Pestov.
\newblock Hyperlinear and sofic groups: a brief guide.
\newblock {\em Bull. Symbolic Logic}, 14(4):449--480, 2008.

\bibitem{Ra08}
Florin R{\u{a}}dulescu.
\newblock The von {N}eumann algebra of the non-residually finite {B}aumslag
  group {$\langle a,b\vert ab^3a^{-1}=b^2\rangle$} embeds into {$R^\omega$}.
\newblock In {\em Hot topics in operator theory}, volume~9 of {\em Theta Ser.
  Adv. Math.}, pages 173--185. Theta, Bucharest, 2008.

\bibitem{Sela}
Z.~Sela.
\newblock Diophantine geometry over groups. {VI}. {T}he elementary theory of a
  free group.
\newblock {\em Geom. Funct. Anal.}, 16(3):707--730, 2006.

\bibitem{Stal}
John~R. Stallings.
\newblock Surfaces in three-manifolds and nonsingular equations in groups.
\newblock {\em Math. Z.}, 184(1):1--17, 1983.

\bibitem{Thom1}
Andreas Thom.
\newblock About the metric approximation of {H}igman's group.
\newblock {\em J. Group Theory}, 15(2):301--310, 2012.

\bibitem{W00}
Benjamin Weiss.
\newblock Sofic groups and dynamical systems.
\newblock {\em Sankhy\=a Ser. A}, 62(3):350--359, 2000.
\newblock Ergodic theory and harmonic analysis (Mumbai, 1999).

\end{thebibliography}
\end{document}